\documentclass[12pt]{article}
\usepackage{setspace}

\usepackage{amssymb}

\usepackage{amsmath}
  \usepackage{paralist}
  \usepackage{graphics} 
  \usepackage{epsfig} 
 \usepackage[colorlinks=true]{hyperref}
\usepackage{ifpdf}

\allowdisplaybreaks[4]

\usepackage{graphicx,amssymb,mathrsfs,latexsym,amsthm}
\usepackage{verbatim}

\newtheorem{theorem}{Theorem}[section]
\newtheorem{lemma}[theorem]{Lemma}

\newtheorem{remark}[theorem]{Remark}
\newtheorem{definition}[theorem]{Definition}
\newtheorem{example}[theorem]{Example}
\renewcommand{\theequation}{\thesection.\arabic{equation}}
\numberwithin{equation}{section}

\begin{document}
\title{Variational Principles For BS Dimension of Subsets\footnote{The work is supported by the National Natural Science Foundation of China (10971100) and National Basic Research Program of China (973 Program) (2007CB814800).}}

\author{Chenwei Wang,$^\dag$  Ercai Chen$^{\dag \ddag}$\\ 
\small \it $\dag$ School of Mathematical Science, Nanjing Normal University\\
\small \it Nanjing 210097, Jiangsu, P.R.China\\
\small \it e-mail: ecchen@njnu.edu.cn\\
\small \it $\ddag$ Center of Nonlinear Science\\
\small \it Nanjing University, Nanjing 210093, Jiangsu, P.R.China\\}
\date{}
\maketitle

\begin{center}
\begin{minipage}{120mm}
{\small {\bf Abstract.} 
\indent We redefine BS-dimension for Carath$\acute{\textrm{e}}$odory structure by packing method. We have the same dimension properties with respect to the cover method and check the Bowen's equation for the new dimension as well. Besides, we consider the relation between the new BS-dimension and upper and lower BS-density respectively. We extend the variational principles of entropy to BS dimension. }
\end{minipage}
\end{center}

\vskip0.5cm {\small{\bf Keywords and phrases} BS-dimension;
Variational principle; Bowen's equation.}\vskip0.5cm

\setcounter{equation}{0}

\section{Introduction}
 Besides the notion of Hausdorff dimension $\dim_H$, another frequently used notion of dimension is the Box dimension.
 For a totally bounded set $E$ in a metric space, its (upper) Box dimension is
 $$\overline{\dim}_BE=\limsup_{\varepsilon\rightarrow 0}\frac{\log N(E,\varepsilon)}{-\log \varepsilon},$$
 where $N(E,\varepsilon)$ denotes the largest possible number of disjoint balls of diameter $\varepsilon$ centered at points of $E$.
 However this notion suffers from the lack of associated measures. Tricot (\cite{Tricot79}, \cite{Tricot}) introduce packing dimension, which is counterpart to Hausdorff dimension, used in measuring fractal dimension of sets. Packing dimension and Hausdorff dimension have many similar natures. For example, both of them have a close relationship with the density\cite{Mattila}.

Throughout this paper, by a topological dynamical system (TDS) ($X$, $f$) we mean a compact metric space $X$ together with a continuous self-map $f:X\rightarrow X$. Let $M(X)$, $M(X, f)$ denote respectively the sets of all Borel probability measures, $f-$invariant Borel probability measures. By a measure theoretical dynamical system (m.t.d.s.) we mean $(X, \mathcal{C},\nu, f)$, where $X$ is a set, $\mathcal{C}$ is a $\sigma-$algebra over $X$, $\nu$ is a probability measure on $\mathcal{C}$ and $f$ is a measure preserving transformation.

In 1958 Kolmogorov \cite{Kolmogorov} associated to any m.t.d.s. $(X,\mathcal{C},\nu,f)$ an isomorphic invariant, namely the measure-theoretical entropy $h_{\nu}(f)$. Later on in 1965, Adler, Konheim and McAndrew \cite{Adler} introduced for any TDS $(X, f)$ an analogous notion of topological entropy $h_{top}(f)$, as an invariant of topological conjugacy. There is a basic relation between topological entropy and measure-theoretic entropy: if $(X,f)$ is a TDS, then
$$h_{top}(f)=\sup\{h_{\mu}(f):\mu \in M(X,f)\}.$$
This variational principle was proved by Goodman \cite{Goodman}, and plays a fundamental role in ergodic theory and dynamcial systems (cf. \cite{Pesin}, \cite{Walters}).

In 1973, Bowen \cite{Bowen2} introduced the topological entropy $h_{top}^{B}(f,Z)$ for any set $Z$ in a TDS $(X,f)$ in a way resembling Hausdorff dimension, which we call Bowen's topological entropy. In particular, $h_{top}^B(f,X)=h_{top}(f)$. Bowen's topological entropy plays a key role in topological dynamics and dimension theory \cite{Pesin}.

To study a nature question whether there is certain variational relation between Bowen's topological entropy and measure-theoretic entropy for arbitrary non-invariant compact set, or Borel set in general. For example, when $K\subset X$ is $f-$invariant but not compact, or $K$ is compact but not $f-$invariant, it may happen that $h_{top}^B(f,K)>0$ but $\mu(K)=0$ for any $\mu\in M(X,f)$. For this purpose, Feng and Huang \cite{Feng} defined measure-theoretic entropy for elements in $M(X)$.

In 2000, Barreira and Schmeling \cite{BS} defined a new dimension called BS-dimension with Carath$\acute{\textrm{e}}$odory structure. The BS-dimension satisfies so called Bowen pressure formula. In this paper, we consider BS-dimension in packing method and get the variational principles.

 \section{Definitions}

  \indent Let ($X$,$d$) be a compact metric space with metric $d$, $f:X \rightarrow X$ a continuous map, and $u:X\rightarrow \mathbb{R}$ a positive continuous function.
 For any $n\in \mathbb{N}$, the $n$-th Bowen metric $d_n$ on $X$ is defined by
$$d_n(x,y)=\max\{d(f^k(x),f^k(y)):k=0, \cdots, n-1\}.$$

For every $\varepsilon>0$ we denote by $B_n(x,\varepsilon)$, $\overline{B}_n(x,\varepsilon)$ the open (resp. closed) ball of radius $\varepsilon$ in the metric $d_n$ around $x$, i.e.,
$$B_n(x,\varepsilon)=\{y\in X:d_n(x,y)<\varepsilon\},$$
$$\overline{B}_n(x,\varepsilon)=\{y\in X:d_n(x,y)\leq \varepsilon\}.$$
 According to the theory of Carath$\acute{\textrm{e}}$odory dimension structure, we consider collection of sets $\mathcal{F}=\{B_n(x, \varepsilon):x\in Z, n\in \mathbb{N}, \varepsilon>0\}$, and functions
  $$\psi(B_n(x,\varepsilon))=\frac{1}{n},$$
  $$\xi(B_n(x,\varepsilon))=1,$$
  $$\eta(B_n(x,\varepsilon))=\sup\limits_{x\in B_n(x,\varepsilon)}\exp (\Sigma_{i=0}^{n-1}u(f^ix)).$$

 For $n \geq 1$, $\varepsilon>0$, we denote
 $$\mathcal{W}_n(\varepsilon)=\{B_n(x,\varepsilon):x\in Z\}.$$
 For convenience, for any $B=B_n(x,\varepsilon)\in \mathcal{W}_n(\varepsilon)$, we call the integer $n(B)=n$ the length of $B$ and $x_B=x$ the center of $B$. For any $B\in \mathcal{F}_{\varepsilon}$, function $u$ can induce a function by
 $$u(B)=\sup_{x\in B}\sum_{i=0}^{n(B)-1}u(f^ix).$$

The following dimension was first defined by Barreira and Schmeling \cite{BS}.

\begin{definition}\cite{BS} If $Z\subset X$. For any $\alpha >0$, $N\in \mathbb{N}$ and $\varepsilon>0$ we define
 $$M(Z, \alpha, \varepsilon, N)=\inf_{\mathcal{G}}\{\sum\limits_{B\in \mathcal{G}}\exp(-\alpha u(B))\},$$
where the infimum is taken over all finite or countable $\mathcal{G}\subset \cup_{j\geq N}\mathcal{W}_j(\varepsilon)$ that cover $Z$. Clearly $M(Z, \alpha, \varepsilon, N)$ is a finite out measure on $X$, and increases as $N$ increases. Define $$m(Z, \alpha, \varepsilon)=\lim_{N\rightarrow \infty}M(Z, \alpha, \varepsilon, N)$$ and
\begin{eqnarray*}
  \dim_{\rm{BSC}}(Z,\varepsilon) &=& \inf\{\alpha: m(Z, \alpha, \varepsilon)=0\} \\
   &=&  \sup\{\alpha: m(Z, \alpha, \varepsilon)=\infty\}.
\end{eqnarray*}

The BS-C dimension is $\dim_{\rm{BSC}}Z=\lim\limits_{\varepsilon \rightarrow 0}\dim_{\rm{BSC}}(Z,\varepsilon)$: the limit exists because given $\varepsilon_1<\varepsilon_2$, we have $m(Z,\alpha, \varepsilon_1)\geq m(Z,\alpha, \varepsilon_2)$, so
$\dim_{\rm{BSC}}(Z,\varepsilon_1)\geq \dim_{\rm{BSC}}(Z,\varepsilon_2)$.
\end{definition}
In the theory of dimension, covering and packing are two ways to obtain dimension. Next, we define a new dimension by packings.

\begin{definition}  If $Z\subset X$. For any $\alpha >0$, $N\in \mathbb{N}$ and $\varepsilon>0$ we define
 $$P(Z, \alpha, \varepsilon, N)=\sup_{\mathcal{G}}\{\sum\limits_{B\in \mathcal{G}}\exp(-\alpha u(B))\},$$
where the supermum is taken over all finite or countable pairwise disjoint families $\{\overline{B}_{n_i}(x_i,\varepsilon)\}$ such that $x_i\in Z$, $n_i\geq N$ for all $i$.
The quantity $P(Z,\alpha, \varepsilon, N)$ does not decrease as $N$, $\varepsilon$ decrease, hence the following limits exist:
$$P^*(Z, \alpha, \varepsilon)=\lim_{N\rightarrow \infty}P(Z, \alpha, \varepsilon, N).$$
Define $$P(Z,\alpha, \varepsilon)=\inf\{\sum_{i=1}^{\infty}P^*(Z_i, \alpha, \varepsilon): \bigcup_{i=1}^{\infty}Z_i \supseteq Z\}.$$
There exists a critical value of the parameter $\alpha$, which we will denote by $\dim_{\rm{BSP}}(Z, \varepsilon), $where $P(Z, \alpha, \epsilon)$ jumps from $\infty$ to $0$, i.e.,
\begin{eqnarray*}
  \dim_{\rm{BSP}}(Z,\varepsilon) &=& \inf\{\alpha: P(Z, \alpha, \varepsilon)=0\} \\
   &=& \sup\{\alpha: P(Z, \alpha, \varepsilon)=\infty\} .
\end{eqnarray*}
Note that  $\dim_{\rm{BSP}}(Z,\varepsilon)$ increases when $\varepsilon$ decreases. We call
$$\dim_{\rm{BSP}}Z=\lim_{\varepsilon \rightarrow 0}\dim_{\rm{BSP}}(Z, \varepsilon)$$
the the BS-Packing (or BS-P) dimension of $Z$.
\end{definition}

After the definitions of dimension, we consider the corresponding definitions of capacity.

\begin{definition}
If $Z\subset X$. For any $\alpha >0$, $N\in \mathbb{N}$ and $\varepsilon>0$ we define
$$R_{C}(Z, \alpha, \varepsilon, N)=\inf_{\mathcal{G}}\{\sum\limits_{B\in \mathcal{G}}\exp(-\alpha u(B))\},$$
where the infimum is taken over all finite or countable $\mathcal{G}\subset \mathcal{W}_N(\varepsilon)$ that cover $Z$. We define
$$\overline{r}_{C}(Z, \alpha, \varepsilon)=\limsup_{N\rightarrow \infty}R_{C}(Z, \alpha, \varepsilon, N)$$ and
\begin{eqnarray*}
  \overline{\rm{Cap}}_{\rm{BSC}}(Z,\varepsilon) &=& \inf\{\alpha: \overline{r}_{C}(Z, \alpha, \varepsilon)=0\} \\
   &=&  \sup\{\alpha: \overline{r}_{C}(Z, \alpha, \varepsilon)=\infty\}.
\end{eqnarray*}

The BS-Capacity is $\overline{\rm{Cap}}_{\rm{BSC}}Z=\lim\limits_{\varepsilon \rightarrow 0}\overline{\rm{Cap}}_{\rm{BSC}}(Z,\varepsilon)$.
\end{definition}

\begin{definition}
If $Z\subset X$. For any $\alpha >0$, $N\in \mathbb{N}$ and $\varepsilon>0$ we define
 $$R_P(Z, \alpha, \varepsilon, N)=\sup_{\mathcal{G}}\{\sum\limits_{B\in \mathcal{G}}\exp(-\alpha u(B))\},$$
where the supermum is taken over all finite or countable pairwise disjoint families $\{\overline{B}_N(x_i,\varepsilon)\}$ such that $x_i\in Z$, for all $i$. We define
$$\overline{r}_P(Z, \alpha, \varepsilon)=\limsup_{N\rightarrow \infty}R_P(Z, \alpha, \varepsilon, N).$$
and
\begin{eqnarray*}
  \overline{\rm{Cap}}_{\rm{BSP}}(Z,\varepsilon) &=& \inf\{\alpha: \overline{r}_P(Z, \alpha, \varepsilon)=0\} \\
   &=& \sup\{\alpha: \overline{r}_P(Z, \alpha, \varepsilon)=\infty\} .
\end{eqnarray*}
The BS-Capacity is $\overline{\rm{Cap}}_{\rm{BSP}}Z=\lim\limits_{\varepsilon \rightarrow 0}\overline{\rm{Cap}}_{\rm{BSP}}(Z,\varepsilon)$.
\end{definition}

The above two definitions of BS-Capacity can be defined in an alternative way. In fact, for $B_n(x, \varepsilon)\in \mathcal{F}$, we can change the function $u(B_n(x,\varepsilon))=\sup_{y\in B_n(x,\varepsilon)}\sum_{k=0}^{n-1}u(f^k y)$ to $u'(B_n(x,\varepsilon))=\sum_{k=0}^{n-1}u(f^k x)$.
$R_{C_1}(Z,\alpha,\varepsilon,N)$ and $R_{P_1}(Z,\alpha,\varepsilon,N)$ respectively instead of $R_{C}(Z, \alpha, \varepsilon, N)$ and $R_{P}(Z, \alpha, \varepsilon, N)$. $\overline{r}_{C_1}(Z, \alpha, \varepsilon)$ and $\overline{r}_{P_1}(Z, \alpha, \varepsilon)$ respectively instead of $\overline{r}_{C}(Z, \alpha, \varepsilon)$ and $\overline{r}_{P}(Z, \alpha, \varepsilon)$.
We denote $\overline{\textrm{Cap}}^1_{\rm{BSC}}Z$ and $\overline{\textrm{Cap}}^1_{\rm{BSP}}Z$ respectively the new dimension from the changed function by the Hausdorff and Packing methods.

\section{Variational Principle for BS-P Dimension}
\begin{lemma}\cite{Mattila}\label{5ball}
Let $(X,d)$ be a compact metric space and $\mathcal{B}=\{B(x_i, r_i)\}_{i\in \mathcal{I}}$ be a family of close (or open) balls in $X$. Then there exists a finite or countable subfamily $\mathcal{B}'=\{B(x_i, r_i)\}_{i\in \mathcal{I}'}$ of pairwise disjoint balls in $\mathcal{B}$ such that
$$\bigcup_{B\in \mathcal{B}}\subset \bigcup_{i\in \mathcal{I}'}B(x_i, 5r_i).$$
\end{lemma}

 In \cite{Falconer2}, a common inequality is $\dim_H Z\leq \dim_P Z\leq \overline{\dim}_B Z$. We attempt to the construct the similar inequality to compare the $\dim_{\rm{BSC}}Z$, $\dim_{\rm{BSP}}Z$, and $\overline{\textrm{Cap}}Z$. For this purpose, we need the following equivalent Capacity definition as well as the common Box dimension.
\begin{lemma}\label{equal}
$\overline{\textrm{Cap}}_{\rm{BSC}}Z=\overline{\textrm{Cap}}_{\rm{BSP}}Z=\overline{\textrm{Cap}}^1_{\rm{BSC}}Z=\overline{\textrm{Cap}}^1_{\rm{BSP}}Z$
\end{lemma}
\begin{proof} First we show $ \overline{\textrm{Cap}}_{\rm{BSC}}Z=\overline{\textrm{Cap}}^1_{\rm{BSC}}Z$.

Clearly, for any $\mathcal{G} \subset \mathcal{W}_N(\varepsilon) $ which covers $Z$,
$$\sum_{B\in \mathcal{G}}\exp(-\alpha u(B))\leq \sum_{B\in \mathcal{G}}\exp(-\alpha\sum_{k=0}^{N-1}u(f^k(x_B))).$$
Since $\mathcal{G}$ is arbitrary, we have
$$\overline{r}_C(Z,\alpha,\varepsilon)\leq \overline{r}_{C_1}(Z,\alpha,\varepsilon).$$
Hence
$$\overline{\textrm{Cap}}_{\rm{BSC}}Z\leq\overline{\textrm{Cap}}^1_{\rm{BSC}}Z.$$

Conversely, we still suppose $\mathcal{G}\subset \mathcal{W}_N(\varepsilon)$ which covers $Z$. Set $\underline{u}=\min_{x\in X}u(x)$ and $\gamma(\varepsilon)=\sup\{|u(x)-u(y)|: d(x,y)<2\varepsilon\}$.
We obtain
\begin{eqnarray*}
  \sum_{B\in \mathcal{G}}\exp(-\alpha\sum_{k=0}^{N-1}u(f^k(x_B))) &\leq& \sum_{B\in \mathcal{G}}\exp(-\alpha u(B)+\alpha N\gamma(\varepsilon)) \\
   &=&  \sum_{B\in \mathcal{G}}\exp(-\alpha u(B)+\alpha\frac{u(B)}{\underline{u}}\gamma(\varepsilon))\\
   &=&  \sum_{B\in \mathcal{G}}\exp(-\alpha (1-\frac{\gamma(\varepsilon)}{\underline{u}})u(B)).
\end{eqnarray*}
Since $\mathcal{G}$ is arbitrary, we have
\begin{equation*}
R_{C_1}(Z,\alpha, \varepsilon, N)\leq R_{C}(Z,\alpha(1-\frac{\gamma(\varepsilon)}{\underline{u}}), \varepsilon, N).
\end{equation*}
Letting $N\rightarrow \infty$,
\begin{equation*}
\overline{r}_{C_1}(Z,\alpha, \varepsilon)\leq \overline{r}_{C}(Z,\alpha(1-\frac{\gamma(\varepsilon)}{\underline{u}}), \varepsilon).
\end{equation*}

Therefore, $$\overline{\textrm{Cap}}^1_{\rm{BSC}}(Z,\varepsilon)(1-\frac{\gamma(\varepsilon)}{\underline{u}})\leq\overline{\textrm{Cap}}_{\rm{BSC}}(Z, \varepsilon).$$
By the uniform continuity of $u$ on $X$, we conclude that
$$\overline{\textrm{Cap}}^1_{\rm{BSC}}Z\leq\overline{\textrm{Cap}}_{\rm{BSC}}Z.$$

Similarly, we have $$\overline{\textrm{Cap}}_{\rm{BSP}}Z=\overline{\textrm{Cap}}^1_{\rm{BSP}}Z.$$

Next we prove $\overline{\textrm{Cap}}^1_{\rm{BSC}}Z\geq\overline{\textrm{Cap}}^1_{\rm{BSP}}Z$.

$\forall \varepsilon>0$, $N\in \mathbb{N}$ there exists a finite or countable pairwise disjoint families $\mathcal{G}=\{\overline{B}_{N}(x_i, \varepsilon)\}$ such that $x_i \in Z$ satisfying
$$\sum_{B\in \mathcal{G}}\exp(-\alpha u'(B))>R_{P_1}(Z, \alpha, \varepsilon, N)-\varepsilon.$$

Fix $\mathcal{F}\subset \mathcal{W}_N(\varepsilon/2)$ which covers $Z$,
For any $B \in \mathcal{G}$, there exists $F\in \mathcal{F}$ such that $x_B\in F$ and every $F$ contains at most one such $x_B$.
Therefore
$$\sum_{B\in \mathcal{G}}\exp(-\alpha u'(B))\leq\sum_{F\in \mathcal{F}}\exp(-\alpha u'(F)).$$
By the arbitrariness of $\mathcal{F}$, we have
$$\sum_{B\in \mathcal{G}}\exp(-\alpha u'(B))\leq R_{C_1}(Z, \alpha, \varepsilon/2, N).$$
Hence $$R_{P_1}(Z, \alpha, \varepsilon, N)-\varepsilon \leq R_{C_1}(Z, \alpha, \varepsilon/2, N).$$
Then we have
$$\overline{\textrm{Cap}}^1_{\rm{BSP}}(Z, \varepsilon)\leq \overline{\textrm{Cap}}^1_{\rm{BSC}}(Z, \varepsilon/2).$$
Letting $\varepsilon \rightarrow 0$, thus
$$\overline{\textrm{Cap}}^1_{\rm{BSP}}Z\leq\overline{\textrm{Cap}}^1_{\rm{BSC}}Z.$$

Finally, we show $\overline{\textrm{Cap}}^1_{\rm{BSC}}Z\leq \overline{\textrm{Cap}}^1_{\rm{BSP}}Z$.
Let $\{\overline{B}_{N}(x_i, \varepsilon)\}_{i\in \mathcal{I}}$ be a family of closed ball in $X$ with centers in $Z$ which covers $Z$. According to Lemma \ref{5ball}, we can find a finite or countable subfamily $\{B'_{N}(x_i, \varepsilon)\}_{i\in \mathcal{I}'}$ of pairwise disjoint balls with centers in $Z$ still cover $Z$ after their radiuses were enlarged by 5 times.
Thus $$R_{C_1}(Z, \alpha, 5.5\varepsilon, N)\leq \sum_{B\in \mathcal{B}}\exp(-\alpha\sum_{k=0}^{N-1}u(f^k(x_B)))\leq R_{P_1}(Z, \alpha, \varepsilon, N).$$
Hence, we have $\overline{\textrm{Cap}}^1_{\rm{BSC}}Z\leq \overline{\textrm{Cap}}^1_{\rm{BSP}}Z$.\\
\end{proof}

\begin{theorem}
$\dim_{\rm{BSP}}Z\leq \overline{\textrm{Cap}}_{\rm{BSC}}Z$ for $Z\subset X$.
\end{theorem}
\begin{proof}
Assume that $\dim_{\rm{BSP}}Z>0$; otherwise there is nothing to prove.
For any $\varepsilon>0$, we assume that $t<s<\dim_{\rm{BSP}}(Z, \varepsilon)$.
If $\overline{\rm{Cap}}_{\rm{BSP}}(Z,\varepsilon)<t$. Then
$$\overline{r}_{P}(Z,t,\varepsilon)=0.$$
According to the definition of capacity, for any $M>0$, there exists $N_0$, whence $N\geq N_0$, we have
$$R_{P}(Z,t,\varepsilon,N)<M.$$
If $\mathcal{G}_n$ is an arbitrary finite pairwise disjoint families $\{\overline{B}_n(x_i,\varepsilon)\}$ such that $x_i\in Z$, $n\geq N_0$ for all $i$. Then
$$\sum_{B\in \mathcal{G}_n}\exp(-tu(B))<M.$$
Since $s<\dim_{\rm{BSP}}(Z,\varepsilon)$, there exists a finite pairwise disjoint families $\mathcal{G}=\{\overline{B}_{n_i}(x_i,\varepsilon)\}$ such that $x_i\in Z$, $n_i\geq N_0$ for all $i$ and
$$\sum_{B\in\mathcal{G}}\exp(-su(B))>\frac{M}{1-\exp((t-s)\underline{u})}.$$
We set $\mathcal{G}_k=\{B\in \mathcal{G}: n(B)=k\}$,
\begin{eqnarray*}
  \sum_{B\in \mathcal{G}}\exp(-su(B)) &=& \sum_{k=0}^{\infty}\sum_{B\in \mathcal{G}_k}\exp(-su(B))\\
   &=& \sum_{k=0}^{\infty}\sum_{B\in \mathcal{G}_k}\exp((-s+t)u(B))\exp(-tu(B)) \\
   &\leq& \sum_{k=0}^{\infty}\sum_{B\in \mathcal{G}_k}\exp((-s+t)\underline{u}k)\exp(-tu(B)) \\
   &\leq& \sum_{k=0}^{\infty}\exp((-s+t)\underline{u}k)M \\
   &=& \frac{M}{1-\exp((t-s)\underline{u})}.
\end{eqnarray*}
Hence
$$\overline{\rm{Cap}}_{\rm{BSC}}(Z,\varepsilon)\geq \dim_{\rm{BSP}}(Z,\varepsilon).$$
Letting $\varepsilon\rightarrow 0$, we have
$$\overline{\rm{Cap}}_{\rm{BSC}}Z\geq \dim_{\rm{BSP}}Z.$$
\end{proof}

\begin{theorem}
 $\dim_{\rm{BSC}}Z\leq \dim_{\rm{BSP}} Z$ for $Z\subset X$.
\end{theorem}
\begin{proof}
For $\varepsilon >0$, we assume that $s>\dim_{\rm{BSP}}(Z, \varepsilon)$, then $P(Z, s, \varepsilon)=0$.
If $Z\subset \cup_{i=1}^{\infty}Z_i$, such that $P^*(Z_i, s, \varepsilon)<\infty$, $\forall i$.
When $N$ is large enough, we have $P(Z_i, s, \varepsilon, N)<\infty$ and
$R_1(Z_i, s, \varepsilon, N)<\infty$.
This means $\overline{\textrm{Cap}}_{\rm{BSP}}(Z_i, \varepsilon)\leq s$,
hence,
\begin{equation*}
\overline{\textrm{Cap}}_{\rm{BSP}}(Z_i, \varepsilon)\leq \dim_{\rm{BSP}}(Z, \varepsilon), \; \forall i.
\end{equation*}
Letting $\varepsilon\rightarrow 0$, we have
$$\overline{\textrm{Cap}}_{\rm{BSP}}Z_i\leq \dim_{\rm{BSP}}Z.$$
Since BS-C dimension is countably stable, we have
\begin{eqnarray*}
  \dim_{\rm{BSC}}Z &\leq& \sup_{i}\overline{\textrm{Cap}}Z_i \\
   &\leq& \dim_{\rm{BSP}}Z.
\end{eqnarray*}
\end{proof}

Barreira and Schmeling \cite{BS} have showed that BS dimension is the unique root of topological pressure function. Theorem \ref{Bowen} will show BS-P dimension the unique root of packing topological pressure function.
In the following, we give the definition of packing topological pressure.

\begin{definition}
If $Z\subset X$. For continuous $g$:$X\rightarrow \mathbb{R}$, $s\geq 0$, $N\in \mathbb{N}$ and $\varepsilon>0$, we define
\begin{equation}
P_p(g,Z,\beta,\varepsilon,N)=\sup_{\mathcal{G}}\sum_{B\in \mathcal{G}}\exp(-\beta n(B)+\sup_{x\in B}\sum_{i=0}^{n(B)-1}g(f^ix)),
\end{equation}
where the supermum is taken over all finite or countable pairwise disjoint families $\mathcal{G}=\{\overline{B}_{n_i}(x_i,\varepsilon)\}$ such that $x_i\in Z$, $n_i\geq N$ for all $i$. $P_p(g,Z,\beta,\varepsilon,N)$ decreases as $N$ increases. So limit
$$P_p^*(g,Z,\beta,\varepsilon)=\lim_{N\rightarrow \infty}P_p(g,Z,\beta,\varepsilon,N).$$
exists. We define
\begin{equation}
P_p(g,Z,\beta,\varepsilon)=\inf\{\sum_{i=0}^{\infty}P_p^*(g,Z_i,\beta,\varepsilon): \cup_{i=1}^{\infty}Z_i\supset Z\}.
\end{equation}
There exists a critical value of the parameter $\beta$, which we will denote by $P_{Z,p}(g,\varepsilon)$, where $P_p(g,Z,\beta,\varepsilon)$ jumps from $\infty$ to $0$, i.e.
\begin{equation}
P_p(g,Z,\beta,\varepsilon)=\left\{
  \begin{array}{ll}
    0, & \hbox{$\beta>P_{Z,p}(g,\varepsilon)$;} \\
    \infty, & \hbox{$\beta<P_{Z,p}(g,\varepsilon)$.}
  \end{array}
\right.
\end{equation}
We call
\begin{equation}
P_{Z,p}(g)=\lim_{\varepsilon\rightarrow 0}P_{Z,p}(g,\varepsilon)
\end{equation}
the packing topological pressure.
\end{definition}
\begin{theorem}\label{Bowen}
(Bowen pressure formula) We have $\dim_{\rm{BSP}}Z=\alpha$, where $\alpha$ is the unique root of the equation $P_{Z,p}(-\alpha u)=0$.
\end{theorem}
\begin{proof}
For any $\varepsilon>0$, fix $\beta>0$, $N\in \mathbb{N}$, $t\in \mathbb{R}$, $h>0$. For any $n\in \mathbb{N}$ and $x\in Z$, we have
$$\frac{\sum_{k=0}^{n-1}(-u(f^kx))}{n}\leq -\underline{u}<0.$$
Then, for any $A\subset X$,
\begin{eqnarray*}
  &&P_{p}(-(t+h)u,A,\beta,\varepsilon,N) \\
  &=& \sup_{\mathcal{G}}\Big\{\sum_{B\in \mathcal{G}}\exp\big(-\beta n(B)+\sup_{x\in B}\sum_{i=0}^{n(B)-1}-(t+h)u(f^ix)\big)\Big\} \\
   &\leq&  \sup_{\mathcal{G}}\Big\{\sum_{B\in \mathcal{G}}\exp\big(-\beta n(B)+\sup_{x\in B}\sum_{i=0}^{n(B)-1}-hu(f^ix)+\sup_{x\in B}\sum_{i=0}^{n(B)-1}-tu(f^ix)\big)\Big\}\\
   &\leq& \sup_{\mathcal{G}}\Big\{\sum_{B\in \mathcal{G}}\exp\big(- n(B)(\beta+h\underline{u})+\sup_{x\in B}\sum_{i=0}^{n(B)-1}-tu(f^ix)\big)\Big\} \\
   &=&   P_{p}(-tu,A,\beta+h\underline{u},\varepsilon,N).
\end{eqnarray*}
Letting $N\rightarrow \infty$, we have
$$P_{p}^*(-(t+h)u,A,\beta,\varepsilon)\leq P_{p}^*(-tu,A,\beta+h\underline{u},\varepsilon).$$
Hence
$$P_{p}(-(t+h)u,Z,\beta,\varepsilon)\leq P_{p}(-tu,Z,\beta+h\underline{u},\varepsilon).$$
Furthermore
\begin{eqnarray*}
  P_{Z,p}(-(t+h)u,\varepsilon) &=& \sup\{\beta:P_{p}(-(t+h)u,Z,\beta,\varepsilon)=\infty\} \\
   &\leq& \sup\{\beta:P_{p}(-tu,Z,\beta+h\underline{u},\varepsilon)=\infty\}\\
   &=& \sup\{\beta+h\underline{u}:P_{p}(-tu,Z,\beta+h\underline{u},\varepsilon)=\infty\}-h\underline{u} \\
   &=& P_{Z,p}(-tu,\varepsilon)-h\underline{u}.
\end{eqnarray*}
By the arbitrariness of $\varepsilon$, $P_{Z,p}(-tu)$ strictly decreases as $t$ increases.

Suppose $\alpha$ is the root of $P_{Z,p}(-\alpha u)=0$, for any $A\subset X$, we have
 $$P_p^*(-\alpha u,A,0,\varepsilon)=P^*(A,\alpha,\varepsilon).$$
According to the definition of dimension, we also have
$$P_p(-\alpha u,Z,0,\varepsilon)=P(Z,\alpha,\varepsilon).$$
If $\alpha<\dim_{\rm{BSP}}Z$, there exists $\varepsilon_0$, as long as $0<\varepsilon<\varepsilon_0$, we have $\alpha<\dim_{\rm{BSP}}(Z,\varepsilon)$. Hence $$P_p(-\alpha u,Z,0,\varepsilon)=P(Z,\alpha,\varepsilon)=\infty.$$
Therefore $$P_{Z,p}(-\alpha u,\varepsilon)\geq 0,$$
and
$$P_{Z,p}(-\alpha u)\geq 0.$$

If $\alpha>\dim_{\rm{BSP}}Z$, we get $P_{Z,p}(-\alpha u)\leq 0$ by the same reason.

\end{proof}

In \cite{BS}, we can see that if $J$ is a repeller of a topologically mixing $C^1$ expanding map $f$ such that $f$ is conformal on $J$, then for every subset $Z\subset J$ (not necessarily compact or $f$-invariant), we have $\dim_HZ=s$, where $s$ is the unique root of the equation $P_Z(-s\log a)=0$. If we can verify that in symbolic system corresponding dimensions by packing methods are still hold, our dimension $\dim_{\rm{BSP}}$ will be more meaningful.

Let us recall the definition of Packing dimension. If $s\geq 0$, and $\delta>0$, let
\begin{eqnarray*}
  \mathcal{P}^{\alpha}_{\delta}(Z) &=& \sup\{\sum_{i}|B_i|^{\alpha}:\{B_i\} \textrm{are disjoint closed balls with centers in $Z$ and} \\
  && \textrm{  radium not more than}\, \delta\}.
\end{eqnarray*}
Since $\mathcal{P}^{\alpha}_{\delta}(Z)$ decreases when $\delta$ decreases, $\mathcal{P}^{\alpha}_{0}(Z)=\lim_{\delta \rightarrow 0}\mathcal{P}^{\alpha}_{\delta}(Z)$ exists. The $\alpha$-packing measure is defined by
\begin{equation*}
    \mathcal{P}^{\alpha}(Z)=\inf\{\sum_{i}\mathcal{P}^{\alpha}_{0}(Z_i):Z\subset \cup_{i=1}^{\infty}F_i\}.
\end{equation*}
The Packing dimension of $Z$ is defined by
$$\dim_{P}Z=\sup\{s:\mathcal{P}^{\alpha}(Z)=\infty\}=\sup\{s:\mathcal{P}^{\alpha}(Z)=\infty\}.$$

In the following we fix a finite alphabet $\{1,\cdots,L\}$, and endow the sequence space $\Omega=\{1,\cdots,L\}^{\mathbb{N}_0}$ with the usual product topology. Denote the left shift on $\Omega$ by $T$. For $\omega\in \Omega$, let
$$\mathcal{C}_n(\omega)=\{\omega'\in \Omega:\omega'(i)=\omega(i) \textrm{for all} 0\leq i<n\}  \; (n \in \mathbb{N}).$$

 Given a strictly positive continuous function $u: \Omega \rightarrow \mathbb{R}$, associate with it a functional metric $[u]$ on $\Omega$ defined by
\begin{equation*}
    [u](\omega,\omega')=\exp(-S^*_nu(\omega)),
\end{equation*}
where $n=\min\{i\in \mathbb{N}_0:\omega(i)\neq \omega'(i)\}$ and
$$S^*_nu(\omega)=\min_{\omega'\in \mathcal{C}_n(\omega)}S_nu(\omega') \;(\omega\in \Omega, n\in \mathbb{N}).$$

\begin{theorem}
In symbolic system, $u:\Omega \rightarrow \mathbb{R}$ is a strictly positive continuous function, $Z\subset \Omega$, then $\dim_{\rm{BSP}}Z=\dim_PZ$, where $\dim_PZ$ depends on the metric $[u]$.
\end{theorem}

\begin{proof}
For any $\varepsilon>0$ and any $A\subset Z$. Suppose $\mathcal{G}$ is a finite or countable pairwise disjoint families $\{\overline{B}_{n_i}(x_i,\varepsilon)\}$ such that $x_i\in A$, $n_i\geq N$ for all $i$, then $\mathcal{G}$ is a family with diameter less than $\exp(-N\underline{u})$.
$$\mathcal{P}_{\exp(-N\underline{u})}^{\alpha}(A)\geq P(A,\alpha,\varepsilon,N).$$
Letting $N\rightarrow \infty$, then
$$\mathcal{P}_0^{\alpha}(A)\geq P^*(A,\alpha,\varepsilon).$$
Hence,
$$\mathcal{P}^{\alpha}(Z)\geq P(Z,\alpha,\varepsilon).$$
Therefore
$$\dim_PZ\geq \dim_{\rm{BSP}}(Z,\varepsilon).$$
By the arbitrariness of $\varepsilon$, we have $$\dim_PZ\geq \dim_{\rm{BSP}}Z.$$

On the contrary, for any $\varepsilon>0$ and any $A\subset Z$. Suppose $\mathcal{G}$ is a finite or countable pairwise disjoint closed ball family with centers in A and diameter less than $\delta$. For any $B\in \mathcal{G}$, we can find a closed bowen ball $B'\subset B$ when $n(B)$ is large enough and
$$|B|\leq \exp(-u(B')+\gamma(\varepsilon)).$$
Hence for any $\alpha>0$, and a sufficient large $N$, we have
\begin{eqnarray*}
  \sum_{B\in \mathcal{G}} |B|^{\alpha} &\leq&  \sum_{B'}\exp(-\alpha u(B')+\alpha n(B')\gamma(\varepsilon))\\
   &=&   \sum_{B'}\exp(-\alpha u(B')+\alpha \frac{u(B')}{\underline{u}}\gamma(\varepsilon))\\
   &=&   \sum_{B'}\exp(-\alpha (1-\frac{\gamma(\varepsilon)}{\underline{u}})u(B'))\\
   &\leq& P(A,\alpha (1-\frac{\gamma(\varepsilon)}{\underline{u}}),\varepsilon, N)
\end{eqnarray*}
Letting $N\rightarrow \infty$ and by the arbitrariness of $\mathcal{G}$, we have
$$\mathcal{P}_{\delta}^{\alpha}(A)\leq P^*(A,\alpha(1-\frac{\gamma(\varepsilon)}{\underline{u}}),\varepsilon).$$
Letting $\delta \rightarrow 0$ and by the definition of Packing and BS-Packing dimension, we have
$$\mathcal{P}^{\alpha}(Z)\leq P(Z,\alpha(1-\frac{\gamma(\varepsilon)}{\underline{u}}),\varepsilon).$$
Hence
$$\dim_{P}Z(1-\frac{\gamma(\varepsilon)}{\underline{u}})\leq \dim_{\rm{BSP}}(Z,\varepsilon).$$
Since $\varepsilon$ is arbitrary, we have
$$\dim_{P}Z\leq \dim_{\rm{BSP}}Z.$$
\end{proof}

\begin{definition}
Let $\mu \in M(X)$. The measure-theoretical lower and upper BS-dimensions of $\mu$ are defined respectively by
$$\underline{P}_{\mu}(f)=\int \underline{P}_{\mu}(f,x)d\mu(x),\quad \overline{P}_{\mu}(f)=\int \overline{P}_{\mu}(f,x)d\mu(x),$$
  where$$\underline{P}_{\mu}(f,x)=\lim_{\varepsilon\rightarrow 0}\liminf_{n\rightarrow \infty}-\frac{\log\mu(B_n(x,\varepsilon))}{\sum_{i=0}^{n-1}u(f^ix)},$$
  $$\overline{P}_{\mu}(f,x)=\lim_{\varepsilon\rightarrow 0}\limsup_{n\rightarrow \infty}-\frac{\log\mu(B_n(x,\varepsilon))}{\sum_{i=0}^{n-1}u(f^ix)}.$$
\end{definition}
\begin{example} If $u=1$, then $\underline{P}_{\mu}(f,x)=\underline{h}_{\mu}(f,x)$, $\overline{P}_{\mu}(f,x)=\overline{h}_{\mu}(f,x)$, where $\underline{h}_{\mu}(f,x)$, and $\overline{h}_{\mu}(f,x)$ are defined by Brin and Katok \cite{Brin}. They proved that for any $\mu\in M(X,f)$, $\underline{h}_{\mu}(f,x)=\overline{h}_{\mu}(f,x)$ for $\mu-$a.e $x\in X$, and $\int \underline{h}_{\mu}(f,x)d\mu(x)=h_{\mu}(f)$. Hence for $\mu\in M(X,f)$,
$$\underline{h}_{\mu}(f)=\overline{h}_{\mu}(f)=h_{\mu}(f).$$
\end{example}

In the following, we will formulate the variational principles of BS-Packing dimension. To this results, we need to introduce an additional notion. A set in a metric space is said to be analytic if it is a continuous image of the set $\mathcal{N}$ of infinite sequences of natural numbers (with its product topology). It is known that in a Polish space, the analytic subsets are closed under countable unions and intersections, and any Borel set is analytic (cf. Federer \cite{Federer}).
\begin{remark}
By the proof of Lemma \ref{equal}, we can see that if we the change the function $u(B)$ to $u'(B)=\sum_{k=0}^{n(B)-1}u(f^kx_{B})$, the dimension remains unchanged.
\end{remark}
\begin{lemma}\label{ab}
Let $Z\subset X$ and $s$, $\varepsilon>0$. Assume $P^*(Z, s, \varepsilon)=\infty$. Then for any given finite interval $(a,b)\subset \mathbb{R}$ with $a\geq 0$ and any $N\in \mathbb{N}$, there exists a finite disjoint collection $\{\overline{B}_{n_i}(x_i, \varepsilon)\}$ such that $x_i\in Z$, $n_i\geq N$ and $\sum_{i}e^{-\sum_{k=0}^{n_i-1}u(f^kx_i)s}\in (a,b)$.
\end{lemma}
\begin{proof} Take $N_1>N$ large enough such that $e^{-N_1\underline{u}s}<b-a$. Since $P^*(Z, s, \varepsilon)=\infty$, we have $P(Z, s, \varepsilon, N_1)=\infty$. Thus there is a finite disjoint collection $\{\overline{B}_{n_i}(x_i, \varepsilon)\}$ such that $x_i\in Z$, $n_i\geq N_1$ and $\sum_{i}e^{-\sum_{k=0}^{n_i-1}u(f^kx_i)s}>b$. Since $e^{-\sum_{k=0}^{n_i-1}u(f^kx_i)s}\leq e^{-N_1\underline{u}s}<b-a$, by discarding elements in this collection one by one until we can have $\sum_{i}e^{-\sum_{k=0}^{n_i-1}u(f^kx_i)s}\in (a,b)$.
\end{proof}

In \cite{Feng} Feng and Huang proved variational principles for topological entropy and packing topological entropy. In the following we consider BS-P dimension and BS-C dimension.
\begin{theorem}
Let $(X, f)$ be a TDS.\\
(i)If $K\subset X$ is non-empty and compact, then
$$\dim_{\rm{BSP}}K=\sup\{\overline{P}_{\mu}(f): \mu \in M(X), \mu(K)=1\}.$$
(ii)If $Z\subset X$ is analytic, then
$$\dim_{\rm{BSP}}Z=\sup\{\dim_{\rm{BSP}}(K): K\subset Z \textrm{ is compact}\}.$$
\end{theorem}
\begin{proof} We divide the proof into two parts:

 Part 1. $\dim_{\rm{BSP}}Z\geq\sup\{\overline{P}_u(f):\mu\in M(X), \mu(Z)=1\}$ for any Borel set $Z\subset X$.

To see this, let $\mu \in M(X)$ with $\mu(Z)=1$ for some Borel set $Z\subset X$. We need to show that $\dim_{\rm{BSP}}Z\geq \overline{P}_{\mu}(f)$. For this purpose we may assume $\overline{P}_{\mu}(f)>0$; otherwise we have nothing to prove. Let $0<s<\overline{P}_{\mu}(f)$. Then there exist $\varepsilon, \delta>0$, and a Borel set $A\subset Z$ with $\mu(A)>0$ such that
$$\overline{P}_{\mu}(f,x, \varepsilon)>s+\delta,\quad \forall x\in A,$$
where $\overline{P}_{\mu}(f,x, \varepsilon)=\limsup_{n\rightarrow \infty}-\frac{\log\mu(B_n(x,\varepsilon))}{\sum_{i=0}^{n-1}u(f^ix)}$.

Next we show that $P(Z,s,\varepsilon/5)=\infty$, which implies that $\dim_{\rm{BSP}}Z\geq\dim_{\rm{BSP}}(Z,\varepsilon/5)\geq s$.
To achieve this, it suffices to show that $P^*(E,s, \varepsilon/5)=\infty$ for any Borel $E\subset A$ with $\mu(E)>0$. Fix such a set $E$. Define
$$E_n=\{x\in E: \mu(B_n(x,\varepsilon))<e^{-\sum_{i=0}^{n-1}u(f^ix)(s+\delta)}\},\qquad n\in \mathbb{N}.$$
Since $E\subset A$, we have $\bigcup_{n=N}^{\infty}E_n=E$ for each $N\in \mathbb{N}$. Fix $N\in \mathbb{N}$. Then $\mu(\bigcup_{n=N}^{\infty}E_n)=\mu(E)$, and hence there exists $n\geq N$ such that
$$\mu(E_n)\geq \frac{1}{n(n+1)}\mu(E).$$
Fix such $n$ and consider the family $\{B_n(x,\varepsilon/5):x\in E_n\}$. By Lemma \ref{5ball} (in which we use $d_n$ instead of $d$), there exits a finite pairwise disjoint family $\{B_n(x_i, \varepsilon/5)\}$ with $x_i\in E_n$ such that
$$\bigcup_iB_n(x_i,\varepsilon)\supset \bigcup_{x\in E_n}B_n(x, \varepsilon/5)\supset E_n.$$
Hence \begin{eqnarray*}
        P(E,s,\varepsilon/5, N) &\geq& P(E_n,s,\varepsilon/5,N)\geq \sum_ie^{-\sum_{k=0}^{n-1}u(f^kx_i)s} \\
         &\geq& e^{\sum_{k=0}^{n-1}u(f^kx_i)\delta}\sum_ie^{-\sum_{k=0}^{n-1}u(f^kx_i)(s+\delta)} \\
         &\geq& e^{\sum_{k=0}^{n-1}u(f^kx_i)\delta}\sum_i\mu(B_n(x_i,\varepsilon))\\
         &\geq& e^{\sum_{k=0}^{n-1}u(f^kx_i)\delta}\mu(E_n)\geq  e^{\sum_{k=0}^{n-1}u(f^kx_i)\delta}\frac{\mu (E)}{n(n+1)}
      \end{eqnarray*}
Since $\frac{e^{\sum_{k=0}^{n-1}u(f^kx_i)\delta}}{n(n+1)}\rightarrow \infty$ as $n\rightarrow \infty$, letting $N\rightarrow \infty$ we obtain that $P^*(E,s,\varepsilon/5)=\infty$.\\

Part 2. Let $Z \subset X$ be analytic with $\dim_{\rm{BSP}}Z>0$. For any $0<s<\dim_{\rm{BSP}}Z$, there exists a compact set $K\subset Z$ and $\mu\in M(K)$ such that $\overline{P}_{\mu}(f)\geq s$.

Since $Z$ is analytic, there exists a continuous surjective map $\phi:\mathcal{N}\rightarrow Z$. Let $\Gamma_{n_1,n_2,\cdots,n_p}$ be the set of $(m_1,m_2,\cdots)\in \mathcal{N}$ such that $m_1\leq n_1$, $m_2\leq n_2$, $\cdots$, $m_p\leq n_p$ and let $Z_{n_1,\cdots,n_p}$ be the image of $\Gamma_{n_1,\cdots,n_p}$ under $\phi$.

Take $\varepsilon>0$ small enough so that $0<s<\dim_{\rm{BSP}}(Z,\varepsilon)$. Take $t\in(s, \dim_{\rm{BSP}}(Z,\varepsilon))$. We are going to construct inductively a sequence of finite sets $(K_i)_{i=1}^{\infty}$ and a sequence of finite measures $(\mu_i)_{i=1}^{\infty}$ so that $K_i\subset Z$ and $\mu_i$ is supported on $K_i$ for each $i$. Together with these two sequences, we construct also a sequence of integers $(n_i)$, a sequence of positive numbers $(\gamma_i)$ and a sequence of integer-valued function $(m_i:K_i\rightarrow \mathbb{N})$. The method of our construction is inspired by the work of Joyce and Preiss \cite{Joyce}, Feng and Huang \cite{Feng}.

The construction is divided into several small steps:\\
Step 1. Construct $K_1$ and $\mu_1$, as well as $m_1(\cdot)$, $n_1$ and $\gamma_1$.\\
Note that $P(Z,t,\varepsilon)=\infty$. Let
$$H=\bigcup\{G\subset X: G \textrm{ is open, } P(Z\cap G, t, \varepsilon)=0\}.$$
Then $P(Z\cap H, t,\varepsilon)=0$ by the separability of $X$. Let $Z'=Z \backslash H=Z\cap (X\backslash H)$. For any open set $G\subset X$, either $Z'\cap G=\emptyset$, or $P(Z'\cap G,t,\varepsilon)>0$. To see this, assume $P(Z'\cap G,t,\varepsilon)=0$ for an open set $G$; then $P(Z\cap G,t, \varepsilon)\leq P(G\cap Z',t, \varepsilon)+P(Z\cap H,t, \varepsilon)=0$, implying $G\subset H$ and hence $Z'\cap G=\emptyset$.

Note that $P(Z',t,\varepsilon)=P(Z,t,\varepsilon)=\infty$ (because $P(Z,t,\varepsilon)\leq P(Z',t,\varepsilon)+P(Z\cap H,t,\varepsilon)=P(Z',t,\varepsilon)$). It follows $P(Z',s,\varepsilon)=\infty$. By Lemma \ref{ab}, we can find a finite set $K_1\subset Z'$, an integer-valued function $m_1(x)$ on $K_1$ such that the collection $\{\overline{B}_{m_1(x)}(x,\varepsilon)\}_{x\in K_1}$ is disjoint and
$$\sum_{x\in K_1}e^{-\sum_{k=0}^{m_1(x)-1}u(f^kx)s}\in (1,2).$$
Define $\mu_1=\sum_{x\in K_1}e^{-\sum_{k=0}^{m_1(x)-1}u(f^kx)s}\delta_x$, where $\delta_x$ denotes the Dirac measure at $x$. Take a small $\gamma_1>0$ such that for any function $z:K_1\rightarrow X$ with $d(x,z(x))\leq \gamma_1$, we have for each $x\in K_1$,
\begin{equation}\label{4.1}
      \Big(\overline{B}(z(x),\gamma_1)\cup\overline{B}_{m_1(x)}(z(x),\varepsilon)\Big)\cap\Big(\bigcup_{y\in K_1\backslash\{x\}}\overline{B}(z(y),\gamma_1)\cup\overline{B}_{m_1(x)}(z(y),\varepsilon)\Big)=\emptyset.
\end{equation}
Here and afterwards, $\overline{B}(x,\varepsilon)$ denotes the closed ball $\{y\in X: d(x,y)\leq \varepsilon\}$. Since $K_1\subset Z'$, $P(Z\cap B(x,\gamma_1/4),t,\varepsilon)\geq P(Z'\cap B(x,\gamma_1/4),t,\varepsilon)>0$ for each $x\in K_1$. Therefore we can pick a large $n_1\in \mathbb{N}$ so that $Z_{n_1}\supset K_1$ and $P(Z_{n_1}\cap B(x,\gamma_1/4),t,\varepsilon)>0$ for each $x\in K_1$.\\

Step 2. Construct $K_2$ and $\mu_2$, as well as $m_2(\cdot)$, $n_2$ and $\gamma_2$.\\

By (\ref{4.1}), the family of balls $\{\overline{B}(x,\gamma_1)\}_{x\in K_1}$, are pairwise disjoint. For each $x\in K_1$, since $P(Z_{n_1}\cap B(x,\gamma_1/4),t,\varepsilon)>0$, we can construct as Step 1, a finite sets
$$E_2(x)\subset Z_{n_1}\cap B(x,\gamma_1/4).$$
and an integer-valued function
$$m_2:E_2(x)\rightarrow \mathbb{N}\cap [\max \{m_1(y):y\in K_1\},\infty)$$
such that \\

 (2-a) $  P(Z_{n_1}\cap G,t,\varepsilon)>0 \textrm{   for each open set } G \textrm{ with } G\cap E_2(x)\neq \emptyset $;\\

 (2-b) The elements in  $\{\overline{B}_{m_2(y)}(y,\varepsilon)\}_{y\in E_2(x)}$  are disjoint, and\\
 $$\mu_1(\{x\})<\sum_{y\in E_2(x)}e^{\sum_{k=0}^{m_2(y)-1}u(f^ky)s}<(1+2^{-2})\mu_1(\{x\}).$$
To see it, we fix $x\in K_1$. Denote $F=Z_{n_1}\cap B(x,\gamma_1/4)$. Let
$$H_x=\bigcup\{G\subset X: G \textrm{ is open } P(F\cap G,t, \varepsilon)=0\}.$$
Set $F'=F\backslash H_x$. Then as in Step 1, we can show that $P(F',t,\varepsilon)=P(F,t,\varepsilon)>0$ and furthermore, $P(F'\cap G,s, \varepsilon)>0$ for any open set $G$ with $G\cap F'\neq \emptyset$. Note that $P(F',s,\varepsilon)=\infty$ (since $s<t$), by Lemma \ref{4.1}, we can find a finite set $E_2(x)\subset F'$ and a map $m_2:E_2(x)\rightarrow \mathbb{N}\cap [\max\{m_1(y):y \in K_1\},\infty)$ so that (2-b) holds. Observe that if a open set $G$ satisfies $G\cap E_2(x)\neq \emptyset$, then $G\cap F'\neq \emptyset$, and hence $P(Z_{n_1}\cap G, t,\varepsilon)\geq P(F'\cap G, t, \varepsilon)>0$. Thus (2-a) holds.

Since the family $\{\overline{B}(x,\gamma_1)\}_{x\in K_1}$ is disjoint, $E_2(x)\cap E_2(x')=\emptyset$ for different $x$, $x'\in K_1$. Define $K_2=\cup_{x\in K_1}E_2(x)$ and
$$\mu_2=\sum_{y\in K_2}e^{-\sum_{k=0}^{m_2(y)-1}u(f^ky)s}\delta_y.$$
By (\ref{4.1}) and (2-b), the elements in $\{\overline{B}_{m_2(y)}(y,\varepsilon)\}_{y\in K_2}$ are pairwise disjoint. Hence we can take $0<\gamma_2<\gamma_1/4$ such that for any function $z:K_2\rightarrow X$ with $d(x,z(x))<\gamma_2$ for $x\in K_2$, we have
\begin{equation}\label{4.2}
    \Big(\overline{B}(z(x),\gamma_2)\cup\overline{B}_{m_2(x)}(z(x),\varepsilon)\Big)\cap\Big(\bigcup_{y\in K_2\backslash \{x\}} \overline{B}(z(y),\gamma_2)  \cup\overline{B}_{m_2(y)}(z(y),\varepsilon)\Big)=\emptyset
\end{equation}
for each $x\in K_2$. Choose a large $n_2\in \mathbb{N}$ such that $Z_{n_1,n_2}\supset K_2$ and $P(Z_{n_1,n_2}\cap B(x,\gamma_2/4),t,\varepsilon)>0$ for each $x\in K_2$.\\

Step 3. Assume that $K_i$, $\mu_i$, $m_i(\cdot)$, $n_i$ and $\gamma_i$ have been constructed for $i=1,\cdots,p$. In particular, assume that for any function $z:K_p\rightarrow X$ with $d(x,z(x))<\gamma_p$ for $x\in K_p$, we have
\begin{equation}\label{4.3}
    \Big(\overline{B}(z(x),\gamma_p)\cup\overline{B}_{m_p(x)}(z(x),\varepsilon)\Big)\cap\Big(\bigcup_{y\in K_p\backslash \{x\}} \overline{B}(z(y),\gamma_p)  \cup\overline{B}_{m_p(y)}(z(y),\varepsilon)\Big)=\emptyset
\end{equation}
for each $x\in K_p$; and $Z_{n_1,\cdots, n_p}\supset K_p$ and $P(Z_{n_1,\cdots,n_p}\cap B(x,\gamma_p/4),t, \varepsilon)>0$ for each $x\in K_p$. We construct below each term of them for $i=p+1$ in a way similar to Step 2.

Note that the elements in $\{\overline{B}(x,\gamma_p)\}_{x\in K_p}$ are pairwise disjoint. For each $x\in K_p$, since $P(Z_{n_1,\cdots, n_p}\cap B(x,\gamma_p/4),t,\varepsilon)>0$, we can construct as Step 2, a finite set
$$E_{p+1}(x)\subset Z_{n_1,\cdots,n_p}\cap B(x,\gamma_p/4)$$
and an integer-valued function
$$m_{p+1}:E_{p+1}(x)\rightarrow \mathbb{N}\cap[\max\{m_p(y): y\in K_p\},\infty)$$
such that

(3-a) $P(Z_{n_1,\cdots,n_p}\cap G,t,\varepsilon)>0$ for each open set $G$ with $G\cap E_{p+1}(x)\neq \emptyset$;

(3-b) $\{\overline{B}_{m_{p+1}(y)}(y,\varepsilon)\}_{y\in E_{p+1}(x)}$ are disjoint and satisfy
$$\mu_p(\{x\})<\sum_{y\in E_{p+1}(x)}e^{-\sum_{k=0}^{m_{p+1}(y)-1}u(f^ky)s}<(1+2^{-p-1})\mu_{p}(\{x\}).$$
Clearly $E_{p+1}(x)\cap E_{p+1}(x')=\emptyset$ for different $x$, $x'\in K_p$.
Define $K_{p+1}=\bigcup_{x\in K_p}E_{p+1}(x)$ and
$$\mu_{p+1}=\sum_{y\in K_{p+1}}e^{-\sum_{k=0}^{m_{p+1}(y)-1}u(f^ky)s}\delta_y.$$
By (\ref{4.3}) and (3-b), $\{\overline{B}_{m_{p+1}(y)}(y,\varepsilon)\}_{y\in K_{p+1}}$ are disjoint. Hence we can take $0<\gamma_{p+1}<\gamma_p/4$ such that for any function $z:K_{p+1}\rightarrow X$ with $d(x,z(x))<\gamma_{p+1}$, we have for each $x\in K_{p+1}$,
\begin{equation}\label{4.4}
\begin{split}
    \Big(\overline{B}(z(x),\gamma_{p+1})\cup\overline{B}_{m_{p+1}(x)}(z(x),\varepsilon)\Big)\bigcap\Big(
    &\cup_{y\in K_{p+1}\backslash \{x\}} \overline{B}(z(y),\gamma_{p+1})\\
    &\cup\overline{B}_{m_{p+1}(y)}(z(y),\varepsilon)\Big)=\emptyset
\end{split}
\end{equation}
Choose a large $n_{p+1}\in \mathbb{N}$ such that $Z_{n_1,\cdots,n_{p+1}}\supset K_{p+1}$ and
$$P(Z_{n_1,\cdots,n_{p+1}}\cap B(x,\gamma_{p+1}/4),t,\varepsilon)>0$$
for each $x\in K_{p+1}$.

As in the above steps, we can construct by induction the sequences $(K_i)$, $(\mu_i)$, $(m_i(\cdot))$, $(n_i)$ and $(\gamma_i)$. We summarize some of their basic properties as follows:\\
(a) For each $i$, the family $\mathcal{F}_i=\{\overline{B}(x,\gamma_i):x\in K_i\}$ is disjoint. Each element in $\mathcal{F}_{i+1}$ is a subset of $\overline{B}(x,\gamma_i/2)$ for some $x\in K_i$.\\
(b) For each $x\in K_i$ and $z\in \overline{B}(x,\gamma_i)$,
$$\overline{B}_{m_i(x)}(z,\varepsilon)\cap \bigcup_{y\in K_i\backslash\{x\}}\overline{B}(y,\gamma_i)=\emptyset.$$
and
\begin{eqnarray*}
\mu_i(\overline{B}(x,\gamma_i))&=&e^{-\sum_{k=0}^{m_i(x)-1}u(f^kx)s}\\
&\leq& \sum_{y\in E_{i+1}(x)}e^{-\sum_{k=0}^{m_{i+1}(y)-1}u(f^kx)s}\\
&\leq& (1+2^{-i-1})\mu_i(\overline{B}(x,\gamma_i)),
\end{eqnarray*}
where $E_{i+1}(x)=B(x,\gamma_i)\cap K_{i+1}$.\\
The second part in (b) implies,
$$\mu_i(F_i)\leq \mu_{i+1}(F_i)=\sum_{F\in \mathcal{F}_{i+1}:F\subset F_i}\mu_{i+1}(F)\leq (1+2^{-i-1})\mu_i(F_i),\qquad F_i\in \mathcal{F}_i.$$
Using the above inequalities repeatedly, we have for any $j>i$,
\begin{equation}\label{4.5}
    \mu_i(F_i)\leq \mu_j(F_i)\leq \prod_{n=i+1}^{j}(1+2^{-n})\mu_i(F_i)\leq C\mu_i(F_i), \qquad \forall F_i\in \mathcal{F}_i,
\end{equation}
where $C=\prod_{n=1}^{\infty}(1+2^{-n})<\infty$.\\
Let $\tilde{\mu}$ be the limit point of $(\mu_i)$ in the weak-star topology. Let
$$K=\bigcap_{n=1}^{\infty}\overline{\bigcup_{i\geq n}K_i}.$$
Then $\tilde{\mu}$ is supported on $K$. Furthermore
$$K=\bigcap_{n=1}^{\infty}\overline{\bigcup_{i\geq n}K_i}\subset \bigcap_{p=1}^{\infty}\overline{Z_{n_1,\cdots,n_p}}. $$
However by the continuity of $\phi$, we can show that
$$\bigcap_{p=1}^{\infty}Z_{n_1,\cdots,n_p}=\bigcap_{p=1}^{\infty}\overline{Z_{n_1,\cdots,n_p}}$$
 by the applying Cantor's diagonal argument. Hence $K$ is a compact subset of $Z$.
On the other hand, by (\ref{4.5}),
\begin{eqnarray*}
  e^{-\sum_{k=0}^{m_i(x)-1}u(f^kx)s} &=& \mu_i(\overline{B}(x,\gamma_i)) \leq \tilde{\mu}(B(x,\gamma_i))\\
   &\leq& C\mu_i(\overline{B}(x,\gamma_i))=Ce^{-\sum_{k=0}^{m_i(x)-1}u(f^kx)s},\quad \forall x\in K_i.
\end{eqnarray*}

In particular, $1<\sum_{x\in K_1}\mu_1(B(x,\gamma_1))\leq \tilde{\mu}(K)\leq \sum_{x\in K_1}C\mu_1(B(x,\gamma_1))\leq 2C$. Note that $K\subset \bigcup_{x\in K_i}\overline{B}(x,\gamma_i/2)$. By the first part of (b), for each $x\in K_i$ and $z\in \overline{B}(x,\gamma_i)$,
$$\tilde{\mu}(\overline{B}_{m_i(x)}(z,\varepsilon))\leq\tilde{\mu} (\overline{B}(x,\gamma_i/2))\leq Ce^{-\sum_{k=0}^{m_i(x)-1}u(f^kx)s}.$$
For each $z\in K$ and $i \in \mathbb{N}$, $z\in \overline{B}(x,\gamma_i/2)$ for some $x\in K_i$. Hence
$$\tilde{\mu}(B_{m_i(x)}(z,\varepsilon))\leq Ce^{-\sum_{k=0}^{m_i(x)-1}u(f^kx)s}.$$
Define $\mu=\tilde{\mu}/\tilde{\mu}(K)$. Then $\mu\in M(K)$, and for each $z\in K$, there exists a sequence $k_i \uparrow \infty$ such that $\mu(B_{k_i}(z,\varepsilon))\leq Ce^{-\sum_{k=0}^{k_i(z)-1}u(f^kz)s}/\tilde{\mu}(K)$. It follows that $\overline{P}_{\mu}(f)\geq s$.

\end{proof}

\section{Weighted BS Dimension}\label{jiaquan BS dingyi}
For any function $h:X\rightarrow [0,\infty)$, $s\geq 0$, $N\in \mathbb{N}$ and $\varepsilon>0$, define
\begin{equation}\label{wbs}
    W(h,s,\varepsilon,N)=\inf\sum_i c_i\exp(-s\sum_{k=0}^{n_i-1}u(f^kx_i)),
\end{equation}
where the infimum is taken over all finite or countable families $\{(B_{n_i}(x_i,\varepsilon),c_i)\}$, such that $0<c_i<\infty$, $x_i\in X$, $n_i\geq N$ and
\begin{equation}
\sum_i c_i \mathcal{X}_{B_i}\geq h,
\end{equation}
where $B_i=B_{n_i}(x_i,\varepsilon)$, and $\mathcal{X}_A$ denotes the characteristic function of $A$.

For $Z\subset X$, and $h=\mathcal{X}_Z$, we set $W(Z,s,\varepsilon,N)=W(\mathcal{X}_Z,s,\varepsilon,N)$.
The quantity $W(Z,s,\varepsilon,N)$ does not decrease as $N$ increases and $\varepsilon$ decrease, hence the following limits exist:
\begin{equation*}
W(Z,s,\varepsilon)=\lim_{N\rightarrow \infty}W(Z,s,\varepsilon,N),\quad W(Z,s)=\lim_{\varepsilon \rightarrow 0}W(Z,s,\varepsilon).
\end{equation*}
Clearly, there exists a critical value of the parameter $s$, which we will denote by $\dim^{WBS}Z$, where $W(Z,s)$ jumps from $\infty$ to $0$, i.e.
\begin{equation}
W(Z,s)=
\left\{
  \begin{array}{ll}
    0, & \hbox{$s > \dim^{WBS}Z$;} \\
   \infty, & \hbox{$s<\dim^{WBS}Z$.}
  \end{array}
\right.
\end{equation}
We call $\dim^{WBS}Z$ the weighted BS dimension of $Z$.
A more extensive and general treatment can be found in Mattila \cite{Mattila}, Kelly \cite{Kelly} and Federer \cite{Federer}.

\section{Equivalence of $\dim_{\rm{BSC}}$ and $\dim^{WBS}$}
\begin{lemma}\label{5.2.1}
If $Z\subset X$. Then for any $s>0$ and $\varepsilon$, $\delta>0$, we have
\begin{equation}
M(Z,s+\delta,6\varepsilon,N)\leq W(Z,s,\varepsilon,N)\leq M(Z,s,\varepsilon,N).
\end{equation}
when $N$ is large enough.
As a result, $$M(Z,s+\delta)\leq W(Z,s) \leq M(Z,s),$$ $$\dim_{\rm{BSC}}Z=\dim^{WBS}Z.$$
\end{lemma}

\begin{proof}Let $Z\subset X$, $s\geq 0$, $\varepsilon$, $\delta>0$. Taking $h=\mathcal{X}_Z$ and $c_i\equiv1$ in (\ref{wbs}), we see that
$$W(Z,s,\varepsilon,N)\leq M(Z,s,\varepsilon,N),$$
for each $N\in \mathbb{N}$.
In the following, we prove that $$M(Z,s+\delta,6\varepsilon,N)\leq W(Z,s,\varepsilon,N),$$ when $N$ is large enough.

Assume that $N>2$ such that $n^2\exp(-\underline{u}n\delta)\leq 1$ for $n\geq N$. Let $\{(B_{n_i}(x_i,\varepsilon),c_i)\}_{i\in \mathcal{I}}$ be a family so that $\mathcal{I}\subset \mathbb{N}$, $x_i\in X$, $0<c_i<\infty$, $n_i\geq N$ and
\begin{equation}\label{5.2.2}
\sum_ic_i\mathcal{X}_{B_i}\geq \mathcal{X}_Z,
\end{equation}
where $B_i=B_{n_i}(x_i,\varepsilon)$. We show below that
\begin{equation}\label{5.2.3}
M(Z,s+\delta,6\varepsilon,N)\leq \sum_{i\in \mathcal{I}}c_i\exp(-s\sum_{k=0}^{n_i-1}u(f^kx_i)),
\end{equation}
which implies $M(Z,s+\delta,6\varepsilon,N)\leq W(Z,s,\varepsilon,N)$.

Denote $\mathcal{I}_n=\{i\in \mathcal{I}:n_i=n\}$, $\mathcal{I}_{n,k}=\{i\in \mathcal{I}_n:i\leq k\}$ for $n\geq N$ and $k\in \mathbb{N}$.
Write for brevity $B_i=B_{n_i}(x_i,\varepsilon)$, and $5B_i=B_{n_i}(x_i,5\varepsilon)$ for $i\in \mathcal{I}$. We may assume $B_i\neq B_j$ for $i\neq j$. For $t>0$, set
\begin{eqnarray*}
  Z_{n,\,t} &=& \{x\in Z :\sum_{i\in \mathcal{I}_n}c_i\mathcal{X}_{B_i}(x)>t\}, \quad \textrm{and} \\
  Z_{n,\,k,\,t} &=& \{x\in Z :\sum_{i\in \mathcal{I}_{n,\,k}}c_i\mathcal{X}_{B_i}(x)>t\}.
\end{eqnarray*}
We divide the proof of (\ref{5.2.3}) into the following three steps.

Step 1. For each $n\geq N$, $k\in\mathbb{N}$, and $t>0$, there exists a finite set $\mathcal{J}_{n,k,t}\subset \mathcal{I}_{n,k}$ such that the balls $B_i~(i\in \mathcal{J}_{n,k,t})$ are pairwise disjoint, $Z_{n,k,t}\subset \bigcup_{i\in\mathcal{J}_{n,k,t} }5B_i$ and
\begin{equation}
\sum_{i\in \mathcal{J}_{n,k,t}}\exp(-s\sum_{k=0}^{n-1}u(f^kx_i))\leq \frac{1}{t}\sum_{i\in \mathcal{I}_{n,k}}c_i\exp(-s\sum_{k=0}^{n-1}u(f^kx_i)).
\end{equation}

To prove the above result, we adopt the method of Federer ( \cite{Federer}, 2.10.24 ) used in the study of weighted Hausdorff measures ( see also Mattila \cite{Mattila}, Lemma 8.16 ). Since $\mathcal{I}_{n,k}$ is finite, by approximating the $c_i$'s from above, we may assume that each $c_i$ is a positive rational, and then multiplying with a common denominator. We may assume that each $c_i$ is a positive integer.
Let $m$ be the least integer with $m\geq t$. Denote $\mathcal{B}=\{B_i:i\in \mathcal{I}_{n,k}\}$ and define $v:\mathcal{B}\rightarrow \mathbb{Z}$ by $v(B_i)=c_i$.
We define by induction integer-valued function $v_0$, $v_1$, $\cdots$, $v_m$ on $\mathcal{B}$ and sub-families $\mathcal{B}_1$, $\cdots$, $\mathcal{B}_m$ of $\mathcal{B}$ starting with $v_0=v$. Using Lemma \ref{5ball} (in which we take the metric $d_n$ instead of $d$) we find a pairwise disjoint subfamily $\mathcal{B}_1$ of $\mathcal{B}$ such that $\bigcup_{B\in \mathcal{B}}B\subset \bigcup_{B\in \mathcal{B}_1}5B$, and hence $Z_{n,k,t}\subset \bigcup_{B\in \mathcal{B}_1}5B$. Then by repeatedly using Lemma \ref{5ball}, we can define inductively for $j=1,\cdots,m$, disjoint subfamilies $\mathcal{B}_j$ of $\mathcal{B}$, such that
\begin{equation}
\mathcal{B}_j\subset \{B\in \mathcal{B}:v_{j-1}(B)\geq 1\}, \quad Z_{n,k,t}\subset \bigcup_{B\in \mathcal{B}_j}5B,
\end{equation}
and the functions $v_j$ such that
\begin{equation}
v_j(B)=
\left\{
  \begin{array}{ll}
    v_{j-1}(B)-1, & \hbox{µ±$B\in \mathcal{B}_j$;} \\
   v_{j-1}(B), & \hbox{µ±$B\in \mathcal{B}\setminus\mathcal{B}_j$.}
  \end{array}
\right.
\end{equation}
For $j<m$, we have
$$Z_{n,k,t}\subset \{x:\sum_{B\in \mathcal{B}:B\ni x}v_j(B)\geq m-j\}.$$
Thus
\begin{equation}
\begin{split}
&\sum_{j=1}^{m}  \sum_{B\in \mathcal{B}_j}\exp(-s\sum_{k=0}^{n-1}u(f^kx_B))\\
=&\sum_{j=1}^{m}\sum_{B\in \mathcal{B}_j}(v_{(j-1)}(B)-v_j(B))\exp(-s\sum_{k=0}^{n-1}u(f^kx_B))\\
\leq&\sum_{B\in \mathcal{B}}\sum_{j=1}^{m}(v_{j-1}(B)-v_j(B))\exp(-s\sum_{k=0}^{n-1}u(f^kx_B))\\
\leq&\sum_{B\in \mathcal{B}}v(B)\exp(-s\sum_{k=0}^{n-1}u(f^kx_B))\\
=&\sum_{i\in \mathcal{I}_{n,k}}c_i\exp(-s\sum_{k=0}^{n-1}u(f^kx_i)).
\end{split}
\end{equation}
Choose $j_0\in\{1,\cdots,m\}$ so that $\sum_{B\in \mathcal{B}_{j_0}}\exp(-s\sum_{k=0}^{n-1}u(f^kx_B))$ is the smallest. Then
\begin{eqnarray*}
  \sum_{B\in \mathcal{B}_{j_0}}\exp(-s\sum_{k=0}^{n-1}u(f^kx_B)) &=& \frac{1}{m}\sum_{i\in \mathcal{I}_{n,k}}c_i\exp(-s\sum_{k=0}^{n-1}u(f^kx_i)) \\
   &=& \frac{1}{t}\sum_{i\in \mathcal{I}_{n,k}}c_i\exp(-s\sum_{k=0}^{n-1}u(f^kx_i))
\end{eqnarray*}
Hence $\mathcal{J}_{n,k,t}=\{i\in \mathcal{I}:B_i\in \mathcal{B}_{j_0}\}$ is desired.

Step 2. For each $n\geq N$ and $t>0$, we have
\begin{equation}\label{5.2.11}
M(Z_{n,t},s+\delta,6\varepsilon,N)\leq \frac{1}{n^2t}\sum_{i\in \mathcal{I}_n}c_i\exp(-s\sum_{k=0}^{n-1}u(f^kx_i)).
\end{equation}
To see this, assume $Z_{n,t}\neq \emptyset$; otherwise there is nothing to prove. Since $Z_{n,k,t}\uparrow Z_{n,t}$, $Z_{n,k,t}\neq \emptyset$ when $k$ is large enough. Let $\mathcal{J}_{n,k,t}$ be the sets constructed in Step 1. Then $\mathcal{J}_{n,k,t}\neq \emptyset$ when $k$ is large enough. Define $E_{n,k,t}=\{x_i:i\in \mathcal{J}_{n,k,t}\}$. Note that the family of all non-empty subsets of $X$ is compact with respect to the Hausdorff distance (cf. Federer \cite{Federer}, 2.10.21). It follows that there is a subsequence $(k_j)$ of natural numbers and a non-empty compact set $E_{n,t}\subset X$ such that $E_{n,k_j,t}$ converges to $E_{n,t}$ in the Hausdorff distance as $j\rightarrow \infty$.
Since any two points in $E_{n,k,t}$ have a distance (with respect to $d_n$) not less than $\varepsilon$, so do the points in $E_{n,t}$. Thus $E_{n,t}$ is a finite set, moreover, $\#(E_{n,k_j,t})=\#(E_{n,t})$ when $j$ is large enough. Hence
\begin{equation}
\bigcup_{x\in E_{n,t}}B_n(x,5.5\varepsilon) \supset \bigcup_{x\in E_{n,k_j,t}}B_n(x,5\varepsilon)=\bigcup_{i\in \mathcal{J}_{n,k_j,t}}5B_i\supset Z_{n,k_j,t}.
\end{equation}
when $j$ is large enough,
and thus $\bigcup_{x\in E_{n,t}}B_n(x,6\varepsilon)\supset Z_{n,t}$. By the way, since $\#(E_{n,k_j,t})=\#(E_{n,t})$ when $j$ is large enough, we have $$\sum_{x\in E_{n,t}}\exp(-s\sum_{k=0}^{n-1}u(f^kx))\leq \frac{1}{t}\sum_{i\in \mathcal{I}_n}c_i\exp(-s\sum_{k=0}^{n-1}u(f^kx_i)).$$
This forces
\begin{equation}
\begin{split}
M(Z_{n,t},s+\delta,6\varepsilon,N)
&\leq \sum_{x\in E_{n,t}}\exp(-(s+\delta)\sum_{k=0}^{n-1}u(f^kx))\\
&\leq \frac{1}{\exp(\underline{u}n\delta)t}\sum_{i\in \mathcal{I}_n}c_i\exp(-s\sum_{k=0}^{n-1}u(f^kx_i))\\
&\leq \frac{1}{n^2t}\sum_{i\in \mathcal{I}_n}c_i\exp(-s\sum_{k=0}^{n-1}u(f^kx_i)).
\end{split}
\end{equation}
Step 3. For any $t\in (0,1)$, we have
$$M(Z,s+\delta,6\varepsilon,N)\leq \frac{1}{t}\sum_{i\in \mathcal{I}}c_i\exp(-s\sum_{k=0}^{n_i-1}u(f^kx_i)).$$
As a result, (\ref{5.2.3}) holds.
To see this, fix $t\in (0,1)$. Note that $\sum_{n=N}^{\infty}n^{-2}<1$. It follows that $Z\subset \bigcup_{n=N}^{\infty}Z_{n,n^{-2}t}$ from (\ref{5.2.2}). Hence by ($\ref{5.2.11}$) and $M(\cdot,s+\delta,6\varepsilon,N)$ is a outer measure, we have
\begin{equation}
\begin{split}
M(Z,s+\delta,6\varepsilon,N)
&\leq \sum_{n=N}^{\infty}M(Z_{n,n^{-2}t},s+\delta,6\varepsilon,N)\\
&\leq \sum_{n=N}^{\infty}\frac{1}{t}\sum_{i\in \mathcal{I}_n}c_i\exp(-s\sum_{k=0}^{n-1}u(f^kx_i))\\
&=\frac{1}{t}\sum_{i\in\mathcal{I}}c_i\exp(-s\sum_{k=0}^{n_i-1}u(f^kx_i)).
\end{split}
\end{equation}
which finishes the proof of the lemma.
\end{proof}

\section{BS Frostman's Lemma}\label{jiaquan BS dingyi}

To prove variational principle for BS-C dimension, we need the following dynamical BS Frostman's lemma.
\begin{lemma}\label{5.3.1}
Let $K$ be a non-empty compact subset of $X$. Let $s\geq 0 $, $N\in \mathbb{N}$ and $\varepsilon>0$. Suppose that $c=W(K,s,\varepsilon,N)>0$. Then there is a Borel probability measure $\mu$ on $X$ such that $\mu(K)=1$ and
\begin{equation}
\mu(B_n(x,\varepsilon))\leq \frac{1}{c}\exp(-s\sum_{k=0}^{n-1}u(f^kx)),\; \forall x\in X, \; n\geq N.
\end{equation}
\end{lemma}

\begin{proof}Clearly $c<\infty$. We define a function $p$ on the space $C(X)$ of continuous real-valued functions on $X$ by
\begin{equation}
p(f)=(1/c)W(\mathcal{X}_K\cdot f,s,\varepsilon,N),
\end{equation}
where $W(\cdot,s,\varepsilon,N)$ is defined as in (\ref{wbs}).

Let $\textbf{1}\in C(X)$ denote the constant function $\textbf{1}(x)\equiv 1$. It is easy to verify that

(1)$p(f+g)\leq p(f)+p(g)$, $\forall$ $f$, $g\in C(X)$.

(2)$p(tf)=tp(f)$, $\forall$ $t\geq 0$ and $f\in C(X)$.

(3)$p(\textbf{1})=1$, $0\leq p(f)\leq \|f\|_\infty$, $\forall$ $f\in C(X)$, and $p(g)=0$, for $g\in C(X)$, $g\leq 0$.\\
By the Hahn-Banach theorem, we can extend the linear functional $t\mapsto tp(1)$, $t\in \mathbb{R}$ from the subspace of the constant function to a linear functional $L:C(X)\rightarrow \mathbb{R}$ satisfying
\begin{equation*}
L(1)=p(1)=1,
\end{equation*}
\begin{equation*}
-p(-f)\leq L(f)\leq p(f),\; \forall\, f\in C(X).
\end{equation*}
If $f\in C(X)$ with $f\geq 0$, then $p(-f)=0$ and so $L(f)\geq 0$. Hence combining the fact $L(1)=1$, we can use the the Riesz representation theorem to find a Borel probability measure $\mu$ on $X$ such that
\begin{equation*}
L(f)=\int fd\mu, \;\forall \,f \in C(X).
\end{equation*}
Now we show that $\mu(K)=1$. To see this, for any compact set $E\subset X\setminus K$, by Uryson lemma there is $f\in C(X)$ such that $0\leq f\leq 1$, $f(x)=1$ for $x\in E$ and $f(x)=0$ for $x\in K$. Then $f \cdot \mathcal{X}_K\equiv 0$ and thus $p(f)=0$. Hence $\mu(E)\leq L(f)\leq p(f)=0$. This shows $\mu(X\setminus K)=0$, i.e. $\mu(K)=1$.

In the end, we show that
\begin{equation*}
\mu(B_n(x,\varepsilon))\leq \frac{1}{c}\exp(-s\sum_{k=0}^{n-1}u(f^kx)), \; \forall \,x\in X, n\geq N.
\end{equation*}
To see this, for any compact set $E\subset B_n(x,\varepsilon)$, by Uryson lemma, there exists $f\in C(X)$, such that $0\leq f\leq 1$, $f(y)=1$ for $y\in E$ and $f(y)=0$ for $y\in X\setminus B_n(x,\varepsilon)$. Then $\mu(E)\leq L(f)\leq p(f)$. Since $f\cdot \mathcal{X}_K\leq \mathcal{X}_{B_n(x,\varepsilon)}$ and $n\geq N$, we have
$$W(\mathcal{X}_K\cdot f,s,\varepsilon,N)\leq \exp(-s\sum_{k=0}^{n-1}u(f^kx)).$$
and thus $p(f)\leq \frac{1}{c}\exp(-s\sum_{k=0}^{n-1}u(f^kx))$. Therefore
$$\mu(E)\leq \frac{1}{c}\exp(-s\sum_{k=0}^{n-1}u(f^kx)).$$
It follows that
\begin{equation}
\begin{split}
\mu(B_n(x,\varepsilon))&=\sup\{\mu(E):E \textrm{~is a compact subset of~} B_n(x,\varepsilon)\}\\
&\leq \frac{1}{c}\exp(-s\sum_{k=0}^{n-1}u(f^kx)).
\end{split}
\end{equation}
\end{proof}

\section{Variational Principle for BS-C Dimension}
Before we discuss the variational principle, we consider the relation between ``local'' dimension and ``global'' dimension.
\begin{theorem}
Let $\mu$ be a Borel probability measure on $X$, $E$ be a Borel subset of $X$ and $0<s<\infty$.

(1)If $\underline{P}_{\mu}(x)\leq s$ for all $x\in E$, then $\dim_{\rm{BSC}}E\leq s$.

(2)If $\underline{P}_{\mu}(x)\geq s$ for all $x\in E$ and $\mu(E)>0$, then $\dim_{\rm{BSC}}E\geq s$.
\end{theorem}
\begin{proof}
(1)For a fixed $r>0$, since $\underline{P}_{\mu}(x)\leq s$ for all $x\in E$, we have $E=\bigcup_{k=1}^{\infty}E_k$, where
\begin{equation}
E_k=\{x\in E:\liminf_{n\rightarrow \infty}\frac{-\log\mu(B_n(x,\varepsilon))}{\sum_{l=0}^{n-1}u(f^lx)}<s+r, \quad \forall \varepsilon\in (0,\frac{1}{k})\}.
\end{equation}
Now fix $k\geq 1$ and $0<\varepsilon<\frac{1}{5k}$. For each $x\in E_k$, there exists a strictly increasing sequence $\{n_j(x)\}_{j=1}^{\infty}$ such that
\begin{equation*}
\mu(B_{n_j(x)}(x,\varepsilon))\geq \exp(-(s+r)\sum_{l=0}^{n_j-1}(u(f^lx))), \;\forall \,j\geq 1.
\end{equation*}
So for any $N\geq 1$, the set $E_k$ is contained in the union of the sets in the family
$$\mathcal{F}=\{B_{n_j(x)}(x,\varepsilon):x\in E_k,n_j(x)\geq N\}.$$
By Lemma \ref{5ball}, there exists a sub family
$$\mathcal{G}=\{B_{n_i}(x_i,\varepsilon)\}_{i\in I}\subset \mathcal{F},$$ consisting of disjoint balls such that
\begin{equation*}
E_k\subset \bigcup_{i\in I}B_{n_i}(x_i,5\varepsilon),
\end{equation*}
and
\begin{equation*}
\mu(B_{n_i}(x_i,\varepsilon))\geq \exp(-(s+r)\sum_{l=0}^{n_i-1}u(f^lx_i)), \forall i\in I.
\end{equation*}
The index set $I$ is at most countable since $\mu$ is a probability measure and $\mathcal{G}$ is a disjoint family of sets, each of which has positive $\mu$-measure. Therefore, $\{B_{n_i}(x_i,5\varepsilon)\}$ is a covering of $E_k$, and consequently
\begin{equation}
M(E_k,s+r,5\varepsilon,N)\leq \sum_{i\in I}\exp(-(s+r)\sum_{l=0}^{n_i-1}u(f^lx_i))\leq \sum_{i\in I}\mu (B_{n_i}(x_i, \varepsilon))\leq 1,
\end{equation}
where the disjointness of $\{B_{n_i}(x_i,\varepsilon)\}_{i\in I}$ is used in the last inequality. It follows that
\begin{equation}
M(E_k,s+r,5\varepsilon)=\lim_{N\rightarrow \infty}M(E_k,s+r,5\varepsilon,N)\leq 1.
\end{equation}
which implies that $\dim_{\rm{BSC}}(E_k,5\varepsilon)\leq s+r$ for any $0<\varepsilon<\frac{1}{5k}$. Letting $\varepsilon\rightarrow 0$ yields
\begin{equation}
\dim_{\rm{BSC}}E_k\leq s+r,\; \forall\, k\geq 1.
\end{equation}
Since BS-C dimension is countably stable, it follows that
\begin{equation}
\dim_{\rm{BSC}}E=\dim_{\rm{BSC}}(\bigcup_{k=1}^{\infty}E_k)=\sup_{k\geq 1}\{\dim_{\rm{BSC}}E_k\}\leq s+r.
\end{equation}
Therefore, $\dim_{\rm{BSC}}E\leq s$ since $r>0$ is arbitrary.

(2)Fix $r>0$, for each $k\geq 1$, put
\begin{equation}
E_k=\{x\in E: \liminf_{n\rightarrow \infty}\frac{-\log \mu(B_n(x,\varepsilon))}{\sum_{l=0}^{n-1}u(f^lx)}>s-\varepsilon, \; \forall \varepsilon\in (0,\frac{1}{k})\}.
\end{equation}
Since $\underline{P}_{\mu}(x)\geq s$ for all $x\in E$, the sequence $\{E_k\}_{k=1}^\infty$ increases to $E$. So by the continuity of the measure (\cite{Mattila}), we have
\begin{equation}
\lim_{k\rightarrow \infty}\mu(E_k)=\mu(E)>0.
\end{equation}
Then fix some $k\geq 1$ with $\mu(E_k)>\frac{1}{2}\mu(E)$. For each $N\geq 1$, put
\begin{equation}
E_{k,N}=\{x\in E_k:\frac{-\log\mu(B_n(x,\varepsilon))}{\sum_{l=0}^{n-1}u(f^lx)}>s-r, \; \forall\, n\geq N, \varepsilon\in (0,\frac{1}{k})\}.
\end{equation}
Since the sequence $\{E_{k,N}\}_{N=1}^{\infty}$ increases to $E_k$, we may pick an $N^*\geq 1$ such that $\mu(E_{k,N^*})>\frac{1}{2}\mu(E_k)$. Write $E^*=E_{k,N^*}$ and $\varepsilon^*=\frac{1}{k}$. Then $\mu(E^*)>0$, and
\begin{equation}\label{5.3.14}
\mu(B_n(x,\varepsilon))\leq \exp(-(s-r)\sum_{l=0}^{n-1}u(f^lx)), \quad \forall x\in E^*, 0<\varepsilon\leq \varepsilon^*, n\geq N^*.
\end{equation}
Now suppose that $\mathcal{F}=\{B_{n_i}(y_i,\frac{\varepsilon}{2})\}_{i\geq 1}$ is a covering of $E^*$ such that
\begin{equation}
E^*\cap B_{n_i}(y_i,\frac{\varepsilon}{2})\neq \emptyset, \; n_i\geq N\geq N^*,   \forall\, i\geq 1, 0<\varepsilon\leq \varepsilon^*.
\end{equation}
For each $i\geq 1$, there exists an $x_i\in E^*\cap B_{n_i}(y_i,\frac{\varepsilon}{2})$. By the triangle inequality
\begin{equation}
B_{n_i}(y_i,\frac{\varepsilon}{2})\subset B_{n_i}(x_i,\varepsilon).
\end{equation}
In combination with (\ref{5.3.14}), this implies
\begin{equation}
\sum_{i\geq 1}\exp(-(s-r)\sum_{l=0}^{n_i-1}u(f^lx_i))\geq \sum_{i\geq 1}\mu(B_{n_i}(x_i,\varepsilon))\geq \mu(E^*).
\end{equation}
Therefore,
\begin{equation*}
M(E^*,s-r,\varepsilon/2,N)\geq \mu(E^*)>0,\; \forall \,N>N^*,
\end{equation*}
and consequently
\begin{equation}
M(E^*,s-r,\varepsilon/2)=\lim_{N \rightarrow \infty}M(E^*,s-r,\varepsilon/2,N)\geq \mu(E*)>0.
\end{equation}
which implies that $\dim_{\rm{BSC}}(E^*,\varepsilon/2)\geq s-r$. Then we have $\dim_{\rm{BSC}}E^*\geq s-r$ by letting $\varepsilon\rightarrow 0$. It following that
$\dim_{\rm{BSC}}E\geq \dim_{\rm{BSC}}E^*\geq s-r$, and hence $\dim_{\rm{BSC}}E\geq s$ since $r>0$ is arbitrary.
\end{proof}
\begin{theorem}
Let $(X,f)$ be a TDS, if $K\subset X$ is non-empty and compact, then
\begin{equation}
\dim_{\rm{BSC}}K=\sup\{\underline{P}_{\mu}(f): \mu\in M(X),\mu(K)=1\}.
\end{equation}
\end{theorem}

\begin{proof}
We first show that $\dim_{\rm{BSC}}(K)\geq \underline{P}_{\mu}(f)$ for any $\mu\in M(X)$ with $\mu(K)=1$. Let $\mu$ be a given such measure.
For $x\in X$, $n\in \mathbb{N}$ and $\varepsilon>0$, we write
$$\underline{P}_{\mu}(f,x,\varepsilon)=\liminf_{n\rightarrow \infty}-\frac{\log\mu(B_n(x,\varepsilon))}{\sum_{i=0}^{n-1}u(f^ix)}.$$

Clearly, $\underline{P}_{\mu}(f,x,\varepsilon)$ is nonnegative and increases as $\varepsilon$ decreases. Hence by the monotone convergence theorem,
\begin{equation}
\lim_{\varepsilon\rightarrow 0}\int \underline{P}_{\mu}(f,x,\varepsilon)d\mu=\int \underline{P}_{\mu}(f,x)d\mu=\underline{P}_{\mu}(f).
\end{equation}
Thus to show $\dim_{\rm{BSC}}K\geq \underline{P}_{\mu}(f)$, it is sufficient to show
$$\dim_{\rm{BSC}}K\geq \int \underline{P}_{\mu}(f,x,\varepsilon)d\mu,$$
for each $\varepsilon>0$.

Fix $\varepsilon>0$ and $l\in \mathbb{N}$. Denote
\begin{equation}
\gamma(\varepsilon)=\sup\{|u(x)-u(y)|: d(x,y)<2\varepsilon \},
\end{equation}
and
\begin{equation}
u_l=\min\{l,\int \underline{P}_{\mu}(f,x,\varepsilon)d\mu(x)-\frac{1}{l}\}.
\end{equation}
Then there exist a Borel set $A_l\subset X$ with $\mu(A_l)>0$ and $N\in \mathbb{N}$ such that
\begin{equation}\label{5.3.20}
\mu(B_n(x,\varepsilon))\leq \exp(-u_l\sum_{k=0}^{n-1}u(f^kx)), \; x\in A_l, \; n\geq N.
\end{equation}
Now let $\{B_{n_i}(x_i,\varepsilon/2)\}$ be a countable or finite family so that $x_i\in X$, $n_i\geq N$ and
$\bigcup_iB_{n_i}(x_i,\varepsilon/2)\supset K\cap A_l.$
We may assume that for each $i$, $B_{n_i}(x_i,\varepsilon)\cap(K\cap A_l)\neq \emptyset$, and choose $y_i\in  B_{n_i}(x_i,\varepsilon/2)\cap (K\cap A_l)$,
Then by (\ref{5.3.20}),
\begin{equation}
\begin{split}
\sum_{i}\exp(-u_l\sum_{k=0}^{n_i-1}u(f^kx_i)(1-\frac{\gamma(\varepsilon)}{\underline{u}}))
&\geq \sum_{i}\exp(-u_l\sum_{k=0}^{n_i-1}u(f^kx_i)+u_l\gamma(\varepsilon)n_i)\\
&\geq \sum_{i}\exp(-u_l\sum_{k=0}^{n_i-1}u(f^ky_i))\\
&\geq \sum_i\mu(B_{n_i}(y_i,\varepsilon))\\
&\geq \sum_i\mu(B_{n_i}(x_i,\varepsilon/2))\\
&\geq \mu(K\cap A_l)=\mu(A_l)>0.
\end{split}
\end{equation}
It follows that
$$M(K,u_l(1-\frac{\gamma(\varepsilon)}{\underline{u}}),\varepsilon/2,N)\geq M(K\cap A_l,u_l(1-\frac{\gamma(\varepsilon)}{\underline{u}}),\varepsilon/2,N)\geq \mu(A_l).$$
Therefore
$$\dim_{\rm{BSC}}K\geq u_l(1-\frac{\gamma(\varepsilon)}{\underline{u}}).$$
Letting $l\rightarrow \infty$, we have
$$\dim_{\rm{BSC}}K\geq \int \underline{P}_{\mu}(f,x,\varepsilon)(1-\frac{\gamma(\varepsilon)}{\underline{u}}) d\mu.$$ Hence
$$\dim_{\rm{BSC}}K\geq\underline{P}_{\mu}(f).$$

We next show that
$$\dim_{\rm{BSC}}K\leq \sup\{\underline{P}_{\mu}(f):\mu\in M(X),\mu(K)=1\}.$$
 We can assume $\dim_{\rm{BSC}}K>0$; otherwise we have nothing to prove. By Lemma \ref{5.2.1}, $\dim_{\rm{BSC}}K=\dim^{WBS}K$. Let $0<s<\dim^{WBS}K$. Then there exist $\varepsilon>0$ and $N\in \mathbb{N}$ such that
 $$c=W(K,s,\varepsilon,N)>0.$$
  By Lemma \ref{5.3.1}, there exists $\mu \in M(X)$ with $\mu(K)=1$ such that
\begin{equation}
  \mu(B_n(x,\varepsilon))\leq \frac{1}{c}\exp(-s\sum_{k=0}^{n-1}u(f^ks)), \;\forall \,x\in X, n\geq N.
\end{equation}
  Clearly
  $\underline{P}_{\mu}(f,x)\geq \underline{P}_{\mu}(f,x,\varepsilon)\geq s$
 for each $x\in X$ and hence
  $$\underline{P}_{\mu}(f)\geq \int \underline{P}_{\mu}(f,x)d\mu(x)\geq s.$$
This finishes the proof.
\end{proof}





\end{document}